\documentclass[oneside,english]{amsart}
\usepackage[T1]{fontenc}
\usepackage[latin9]{inputenc}
\usepackage{amstext}
\usepackage{amsthm}
\usepackage{amssymb}
\usepackage{todonotes}
\usepackage{enumitem}
\usepackage{stmaryrd}
\usepackage{url}
\usepackage[hidelinks]{hyperref}
\usepackage{tikz}
\usetikzlibrary{arrows.meta, positioning, decorations.pathreplacing, calc}

\makeatletter
\numberwithin{equation}{section}
\numberwithin{figure}{section}
\theoremstyle{plain}
\newtheorem{thm}{\protect\theoremname}
\theoremstyle{plain}
\newtheorem{prop}[thm]{Proposition}

\newtheorem{rem}[thm]{Remark}
\newtheorem{conj}[thm]{Conjecture}

\newtheorem{lem}[thm]{\protect\lemmaname}

\newtheorem{defn}[thm]{Definition}
\newtheorem{cor}[thm]{Corollary}
\makeatother

\usepackage{babel}
\providecommand{\lemmaname}{Lemma}
\providecommand{\theoremname}{Theorem}
\providecommand{\theoremname}{Claim}
\providecommand{\theoremname}{Preposition}

\providecommand{\theoremname}{Definition}

\newcommand{\ohad}[1]{\todo[size=\tiny, color=pink]{Ohad: #1}}
\newcommand{\avichai}[1]{\todo[size=\tiny, color=red]{Avichai: #1}}

\newcommand{\Sn}{S_n}

\newcommand{\E}{\mathbf{E}}
\newcommand{\Emb}{\mathrm{Emb}}

\newcommand{\Irr}{\mathrm{Irr}}

\newcommand{\NN}{\mathbb{N}}

\newcommand{\cyc}{\operatorname{cyc}}
\newcommand{\Src}{\operatorname{Src}}

\newcommand{\fix}{\mathrm{fix}}
\title[Covering numbers of conjugacy classes]{
The Robust Thompson's Conjecture for Alternating Groups}

\author{Nathan Keller}
\thanks{Department of Mathematics, Bar-Ilan University. \texttt{Nathan.Keller@biu.ac.il}. Supported by the Israel Science Foundation (grants no.~2669/21 and 2456/25) and by the Binational US-Israel Science Foundation (grant no.~2024120).}
\author{Noam Lifshitz}
\thanks{Einstein Institute of Mathematics, Hebrew University. \texttt{noamlifshitz@gmail.com}. Supported by the European Research Council (StG no.~101163794), by the Israel Science Foundation (grant no.~1980/22), and by the Binational US-Israel Science Foundation (grant no.~2024120).}
\author{Avichai Marmor}
\thanks{Department of Mathematics, Bar-Ilan University. \texttt{avichai@elmar.co.il}. Partially supported by the Israel Science Foundation (grants no.\ 2669/21 and 2456/25) and by the Bar Ilan University's Presidential Scholarship.}

\author{Ohad Sheinfeld}
\thanks{Einstein Institute of Mathematics, Hebrew University. \texttt{oshenfeld@gmail.com}. Partially supported by the European Research Council (StG no.~101163794).}
\begin{document}

\maketitle

\begin{abstract}
We show that there exists a constant $\epsilon > 0$ such that if $C$ is a conjugacy class in $A_n$ of size at least $|A_n|^{1-\epsilon}$, then $A_n \setminus \{1\} \subseteq C^2$. This proves the robust version of Thompson's conjecture for the alternating group, conjectured by Shalev (Annals of Math., 2009). Our proof combines character bounds and combinatorial cancellation techniques with the recent theory of hypercontractivity for functions over symmetric groups.
\end{abstract}

\section{Introduction}

For a group $G$ and a subset $A \subset G$ that generates $G$, the \emph{covering number} of $A$ is the smallest $k$ such that $A^k=G$. The study of covering finite groups by powers of conjugacy classes has a rich history, going back to the  observation of Gleason in the early 60s that each element of $A_n$ can be written as the product of two maximal cycles~\cite{husemoller1962ramified}. A central direction
in this study is determining the conjugacy classes $C$ whose covering number is $2$, that is, whose square covers the entire group. This direction was motivated by the celebrated Thompson's conjecture~\cite{Kourovka}, which asserts that every non-abelian finite simple group $G$ contains a conjugacy class $C$ such that $C^2 = G$. 

For several decades, the study of covering numbers in finite simple groups focused on analyzing specific structural families of conjugacy classes using combinatorial techniques (e.g.,~\cite{arad1985products,bertram1972even,brenner1978covering,herzog2008covering,vishne1998mixing}). These exact covering questions have become closely related to various major research directions in finite simple groups. In particular, the tools developed to study covering numbers and products of conjugacy classes
played a main role in bounding the diameters of Cayley graphs (\cite{lawther1998diameter,liebeck2001diameters}) and in resolving extremal Waring-type problems~\cite{liebeck2010ore,larsen2011waring}.

Despite the extensive research, showing that $C^2=G$ for 
large families of conjugacy classes $C$ remained out of reach. The situation changed in 2008 with two breakthrough works of Larsen and Shalev~\cite{larsen2009word,larsen2008characters} 

which obtained such covering results for most conjugacy classes in $A_n$ at once.

In their work on Waring-type problems in finite simple groups~\cite{larsen2009word}, Larsen and Shalev showed that for all sufficiently large $n$, every permutation $\sigma\in S_n$ with at most $n^{1/128}$ cycles satisfies $(\sigma^{S_n})^2=A_n.$
Since a randomly chosen  $\sigma \in S_n$ has  
$O(\log n)$ cycles and satisfies $\sigma^{S_n}=\sigma^{A_n}$ with a probability that tends to $1$ as $n \to \infty$, this implies that in $A_n$, for \emph{almost every} $\sigma$, the covering number of $\sigma^{A_n}$ is $2$.
The main tool in the proof is character bounds, obtained using results from representation theory. In~\cite{larsen2008characters}, Larsen and Shalev sharpened their 
character bounds and used them to show that $(\sigma^{S_n})^2=A_n$ holds for any $\sigma \in S_n$ with at most $n^{1/4-\epsilon}$ cycles. 

The fundamental results of Larsen and Shalev show that in $A_n$, a \emph{robust} version of Thompson's conjecture holds: Not only there exists a conjugacy class whose square covers the entire group, but actually, every sufficiently large conjugacy class has this property. This naturally gives rise to the quantitative question: \textit{How large should a conjugacy class be to guarantee that its square covers the entire group?}   

The following conjecture that addresses this question was proposed by Shalev.
\begin{conj}\cite[Conjecture~10.3]{shalev2009word}
\label{conj:main}
There exists an absolute constant $\epsilon>0$ such that for every finite simple group $G$ and every conjugacy class $C$ of $G$, we have 
\[
|C^{2}| \ge \min\{|C|^{1+\epsilon}, |G|-\delta\},
\]
where $\delta=0$ if $C^{-1}=C$ and $\delta=1$ otherwise.
\end{conj}
The `growth' statement of the conjecture, which essentially asserts that $|C^2|$ is much larger than $|C|$ for every `not very large' conjugacy class $C$, was proved by Gill, Pyber, Short and Szab\'{o}~\cite[Proposition~5.2]{gill2013product}. The `covering' statement, which is a robust version of Thompson's conjecture asserting that $|C^2|=|G|-\delta$ holds for any conjugacy class with $|C|>|G|^{1-\epsilon}$, remained open. A related robust covering conjecture regarding unions of conjugacy classes was raised and discussed by Shalev in~\cite[Conjecture~1.5]{shalev2023covering}. 

In recent years, tremendous progress has been made in the study of covering numbers, and more generally, products of conjugacy classes (see, e.g., \cite{lifshitz2023bounds, improved_ls, cutoff_profiles_symmetric_groups}).
Most notably, Thompson's conjecture was proved by Larsen and Tiep~\cite{larsentiepuniform,LT24} for all sufficiently large finite non-abelian simple groups, and by Liao, Wang and Zhang~\cite{LWZ26} for all finite non-abelian simple groups. 
At the same time, there was no significant progress toward a robust version. In $A_n$, there were only two advancements toward Shalev's conjecture since the seminal result of Larsen and Shalev~\cite{larsen2008characters} from 2008. The first 
is~\cite{KLS23}, which asserts that $(\sigma^{S_n})^2=A_n$ holds for any $\sigma \in A_n$ with at most $n^{2/5-\epsilon}$ cycles. The second is a very recent result of Dona~\cite{dona2025}
on products of three conjugacy classes in $A_n$ which implies that if $|C|>|A_n|^{1-\epsilon}$, then $C^2$ covers any conjugacy class of size at least $|A_n|^{1-\epsilon}$. 


\subsection*{Our results}

In this paper we prove the robust Thompson's conjecture in $A_n$.
\begin{thm}\label{thm:main}
    There exists an absolute constant $\epsilon > 0$ such that the following holds for all sufficiently large $n$. If $C$ is a conjugacy class of $A_n$ of size $|C| \ge (n!)^{1-\epsilon}$, then
    $A_n \setminus \{1\} \subseteq C^2$.
\end{thm}
An equivalent formulation of the theorem is that there exists $c>0$ such that for any $\sigma \in A_n$ with at most $cn$ cycles, we have $(\sigma^{A_n})^2\supseteq A_n \setminus \{1\}$.
This result is sharp up to the value of $c$, since Bertram~\cite{bertram1972even} showed that if $\sigma \in A_n$ consists of a single cycle of length $\ell<\lfloor \frac{3n}{4} \rfloor$ and $n-\ell$ fixed points, then $(\sigma^{A_n})^2 \nsupseteq A_n \setminus \{1\}$. In the extremal case $\ell=\lfloor \frac{3n}{4} \rfloor-1$, $\sigma$ has slightly more than $n/4$ cycles and $|\sigma^{A_n}|=|A_n|^{3/4+o(1)}$.  Two remaining open problems are to determine the maximum $\epsilon$ such that $A_n \setminus \{1\}\subseteq C^2$ is guaranteed to hold for all $|C|>|A_n|^{1-\epsilon}$, and to prove the assertion of Theorem~\ref{thm:main} for all $n$ (with a specific value of $\epsilon$).

We note that as was recently observed in~\cite{Sheinfeld26}, Shalev's conjecture does not hold in full 
generality, since for any $\epsilon>0$ and for every sufficiently large even $n$, there exist conjugacy classes $C \subset PSL_n(2)$ such that $|C|>|PSL_n(2)|^{1-\epsilon}$, and yet $C^2 \not \supseteq PSL_n(2) \setminus \{1\}$.
This follows from an old result of Uhlig~\cite[Lemma~2]{Uhlig79}; see~\cite[Lemma~3.2]{Nielsen26} for a short proof. 
More generally, this construction and variants of the constructions 
in~\cite[Sections~6-7]{Sheinfeld26} yield counterexamples to the
`covering' part of Shalev's conjecture for finite simple groups of each classical Lie type, along sequences of finite simple groups whose ranks
tend to infinity.

\subsection*{The techniques used in the proof}
Our proof combines classical character bounds, combinatorial cancellation, and techniques from discrete Fourier analysis --- specifically, hypercontractivity for functions over symmetric groups. The two first components, character bounds and combinatorial cancellation, were used by Larsen and Shalev~\cite{larsen2008characters}. Thus, we begin with presenting them in the way they were used in~\cite{larsen2008characters} and then we explain the additional elements required for obtaining our results.

\subsubsection*{Character Bounds.}
The classical tool for this kind of problems is the Frobenius character formula. For a finite group $G$ and elements $\sigma, \tau \in G$, the number of ways to write $\tau$ as a product of two elements in the conjugacy class $C = \sigma^G$ is given by
\begin{equation} \label{eq:frobenius}
\frac{|C|^2}{|G|} \cdot \sum_{\chi \in \Irr(G)} \frac{\chi(\sigma)^2 \overline{\chi(\tau)}}{\chi(1)}.
\end{equation}
To prove that $\tau \in C^2$, it suffices to show that the contribution of the non-trivial irreducible characters in \eqref{eq:frobenius} is strictly less in absolute value than the contribution of the trivial character (which is $1$). 
Over the past three decades, a vast machinery of character bounds has been developed by Liebeck, Shalev, Larsen, Tiep and others to estimate these sums, in different finite simple groups. In particular, the foundational work of Larsen and Shalev~\cite{larsen2008characters} shows that if $\sigma$ has at most $n^{1/4-\epsilon}$ cycles,
the ratio $\chi(\sigma)^2/\chi(1)$ decays polynomially in the dimension of the character. However, this in itself is not sufficient, since  
$\tau$ may have extremely large character values compared to the dimension $\chi(1)$. This is the case when $\tau$ has many small cycles. Hence, an additional ingredient is needed.

\subsubsection*{Cycle Cancellation}
A standard approach to complement the character method is a combinatorial strategy called \emph{cycle cancellation}. Any permutation decomposes into disjoint cycles. Suppose $\tau$ contains a ``problematic'' subset of cycles (e.g., a small number of fixed points or small cycles) acting on a subset of coordinates $\Omega \subset [n]$. In a cancellation step, one finds $\sigma_1, \sigma_2 \in \sigma^{A_n}$ such that $\sigma_1(\Omega)=\Omega, \sigma_2(\Omega)=\Omega$, and  $\sigma_1 \sigma_2 |_\Omega = \tau |_\Omega$. The corresponding cycles can then be ``canceled'', reducing the problem to covering $\tau|_{[n] \setminus \Omega}$ by $(C')^2$, where $C'$ is the conjugacy class in $A_{n-|\Omega|}$ corresponding to the cycle structure of $\sigma$ without the removed cycles.
Larsen and Shalev~\cite{larsen2008characters} showed that 
for each $\tau$ one can devise a  cancellation such that the reduced problem can be solved by character bounds.

\subsubsection*{Difficulties that arise for smaller conjugacy classes $C$.} The method described above works well for `large' conjugacy classes, i.e., classes that correspond to permutations with at most $n^{1/4-\epsilon}$ cycles. When we consider permutations with a much larger number of cycles, say $\sqrt{n}$, the term 
$\chi(\sigma)^2/\chi(1)$ may be larger than $1$, and thus, the character bound cannot be applied even if $\overline{\chi(\tau)}$ is not large. We overcome this by a combinatorial cancellation, this time using cycles of $\tau$ to cancel ``problematic'' cycles in $\sigma$. However, such a cancellation is not always possible; it appears to fail when $\sigma$ and $\tau$ are `imcompatible', e.g., when $\sigma$ has many fixed points and a few long cycles and $\tau$ has many small cycles whose total size is smaller than the size of each long cycle in $\sigma$, and a few long cycles. These cases cannot be handled by the cycle cancellation results in Dona's recent paper~\cite{dona2025} as well.
Therefore, it seems that an additional ingredient is needed.

\subsubsection*{Globalness and Hypercontractivity}
To overcome the difficulties, we introduce a different type of cancellation, in which $\sigma_1 \sigma_2|_{\Omega} =\tau|_{\Omega}$ as above, but $\sigma_1,\sigma_2$ do not preserve $\Omega$. The reduced problem obtained from the cancellation does not pertain to conjugacy classes in $S_{[n] \setminus \Omega}$, but rather to subsets of the set of bijections from $[n] \setminus \Omega$ to $[n] \setminus \Omega'$, for some $\Omega' \subset [n]$ with $|\Omega'|=|\Omega|$. This set of bijections can be combinatorially identified with $S_{[n] \setminus \Omega}$, but the structure of conjugacy classes is lost, so character bounds cannot be applied anymore. We replace them with an analytic tool --- hypercontractivity for global functions over the symmetric group, developed by Keevash and Lifshitz~\cite{keevash2023sharp} (see also~\cite{filmus2020hypercontractivity}). 

Consider the characteristic function
of the reduced set $C' \subset S_{[n]\setminus\Omega}$, which is no longer a conjugacy class. We decompose this function according to the irreducible representations: for each irreducible representation $\rho$, we consider the part of $1_{C'}$ lying in the span of the matrix coefficients of $\rho$. In the conjugacy class setting, this part is a scalar multiple of the corresponding character. Here, it is a linear combination of the matrix coefficients of $\rho$. The hypercontractive inequality bounds the $L^2$-norm of this part of $1_{C'}$ for each $\rho$, and this replaces the character bounds used by Larsen and Shalev. 
However, the inequality holds only for functions that are \emph{global} -- i.e., functions for which no restriction to the family $\{\sigma \in A_n:\sigma(i_1)=j_1,\ldots,\sigma(i_t)=j_t\}$ for a small $t$ changes the expectation significantly. Hence, we develop a combinatorial argument to show that the functions achieved after the application of cancellation are global, and then we apply the hypercontractive inequality to complete the argument.

\subsubsection*{Relation to previous work} In addition to the methods described above, we use the recent result of Dona~\cite{dona2025} to deduce that if both $\sigma$ and $\tau$ have at most $c'n$ cycles for a sufficiently small constant $c'$, then $\tau$ is covered by $(\sigma^{A_n})^2$. Indeed, in this case $|\sigma^{A_n}|>|A_n|^{1-\epsilon'}$, and thus by Dona's result, $(\sigma^{A_n})^2$ covers every conjugacy class of size at least $|A_n|^{1-\epsilon'}$.

The idea of using hypercontractivity for our problem was proposed in the paper~\cite{KLS23}, written by a subset of the authors of this paper. There, it was used to show that  $C^2\supseteq A_n \setminus \{1\}$ holds for any conjugacy class $C=\sigma^{A_n}$ where $\sigma$ has at most $n^{2/5-\epsilon}$ cycles, improving the $n^{1/4-\epsilon}$ bound of Larsen and Shalev~\cite{larsen2008characters} but still remaining far from the linear regime. In this paper, we develop this tool to its full strength, to handle the conjugacy classes of all permutations with at most $cn$ cycles, for a universal constant $c$.

\subsection*{Organization of the paper} In Section~\ref{sec:previous} we introduce definitions and previous results that will be used throughout the paper. In Section~\ref{sec:globalness} we study restrictions of conjugacy classes to families of the form $\{\sigma \in A_n:\sigma(i_1)=j_1,\ldots,\sigma(i_t)=j_t\}$ and establish the globalness statements required in the proof of Theorem~\ref{thm:main}. In Section~\ref{sec:proof} we  present the hypercontractivity argument and the proof of Theorem~\ref{thm:main}.

\section{Definitions and Previous Results}
\label{sec:previous}

In this section we present definitions and notations that will be used throughout the paper and previous results we make use of. 

\subsubsection*{Notations.} Throughout, $c,c_1,c_2,\ldots$~denote constants that are assumed to be `sufficiently small', and 
$C,C_1,C_2,\ldots$~denote constants that are assumed to be sufficiently large.

\subsection{Character bounds}

We make significant use of several character bounds from works of Larsen, Liebeck, Shalev, and Tiep.

The first three results we use are from the aforementioned work of Larsen and Shalev~\cite{larsen2008characters} which proved (among various other results) that for any $\sigma \in A_n$ with at most $n^{1/4-\epsilon}$ cycles, $(\sigma^{A_n})^2 \supseteq A_n \setminus \{1\}$. We shall need the following notion defined in~\cite{larsen2008characters}.

\begin{defn}
    For $\sigma \in S_n$, let $f_\sigma(i)$ be the number of $i$-cycles in the cycle decomposition of $\sigma$. The \emph{orbit growth sequence} of $\sigma$, $e_1,\ldots,e_n$, is defined via the equality
    \[
    e_1+\ldots+e_k=\max \left(\frac{\log(\sum_{i=1}^k i \cdot f_\sigma(i))}{\log n},0 \right),
    \]
    for each $k=1,2,\ldots,n$. The function $E(\sigma)$ is defined as $E(\sigma)=\sum_{i=1}^n \frac{e_i}{i}$.
\end{defn}

\begin{thm}[\cite{larsen2008characters}, Corollary~1.11]
\label{thm:Larsen-Shlaev}
For any $\epsilon>0$, there exists $N$ such that for all $n\ge N$ and all $\sigma\in S_{n}$, 
$E(\sigma)\le 1/4-\epsilon$ implies $(\sigma^{S_{n}})^{2}=A_{n}$.
\end{thm} 


\begin{thm}[\cite{larsen2008characters}, Theorem~1.2]\label{lem:larsen-shalev-1}
Let $\sigma\in S_{n}$ and $\chi\in \mathrm{Irr}(S_{n})$. If $\sigma$ is fixed-point-free, or has $n^{o(1)}$ fixed points, then
    $$|\chi(\sigma)|\le\chi(1)^{1/2+o(1)}.$$
\end{thm}

\begin{lem}[\cite{KLS23}, Lemma~2.6]\label{thm:Larsen-shalev-small-cycles}
For any $\epsilon>0$ and any $m \in \mathbb{N}$, there exist $\delta>0$ and $n_0 \in \mathbb{N}$ such that the following holds. Let $n>n_0$ and suppose that $\sigma\in S_n$ has at most $n^{\delta}$ $i$-cycles for each $i < m$. Then $E(\sigma) \le 1/m +\epsilon.$   
\end{lem}

We use the following result of Larsen and Tiep~\cite{larsen2023squares} for conjugacy classes that split:
\begin{thm}[\cite{larsen2023squares}, Theorem~3]\label{thm:Sn_VS_An}
For all sufficiently large $n$, if $x,y \in \mathsf{A}_n$ are elements which are conjugate in $\mathsf{S}_n$ but $x^{\mathsf{A}_n} \neq x^{\mathsf{S}_n}$, then $x^{\mathsf{A}_n} y^{\mathsf{A}_n} \supseteq \mathsf{A}_n \setminus \{1\}$.
\end{thm}
\begin{rem}\label{rem:Sn_VS_An}
    Theorem~\ref{thm:Sn_VS_An} immediately implies Theorem~\ref{thm:main} for conjugacy classes $C = \sigma^{A_n}$ such that $\sigma^{A_n} \ne \sigma^{S_n}$. Thus, we only need to consider the case $C = \sigma^{A_n} = \sigma^{S_n}$ (for which Theorem~\ref{thm:main} is equivalent to $C^2 = A_n$), so in the sequel we use $\sigma^{A_n},\sigma^{S_n}$ interchangeably.
\end{rem}

We also use a classical result of Liebeck and Shalev~\cite{liebeck2004fuchsian} on the so-called Witten zeta function.

\begin{defn}
    The Witten zeta function for a finite group $G$ is defined by 
    \[
    \zeta(s) = \zeta_G(s) = \sum_{\chi \in \hat{G}} \chi(1)^{-s}.
    \]    
\end{defn}

\begin{thm}[\cite{liebeck2004fuchsian}, Theorem~1.1 and Corollary~2.7] \label{thm:witten zeta function}
For any $\epsilon,s>0$ there exists $n_0$ such that for any $n>n_0,$ we have 
\[2-\epsilon \le \sum_{\chi \in \widehat{S_n}}\chi(1)^{-s} \leq 2+\epsilon, \qquad \mbox{and} \qquad 1-\epsilon \le \sum_{\chi \in \widehat{A_n}}\chi(1)^{-s} \leq 1+\epsilon.
\]
\end{thm}

\subsection{Combinatorial cancellation}

For the combinatorial cancellation, we use several results from the recent work of Dona~\cite{dona2025}, from a work of Bertram~\cite{bertram1972even} and from our previous paper~\cite{KLS23}.

The following theorem of Dona~\cite{dona2025} allows covering all $\tau$'s whose conjugacy class is sufficiently large. We note that this case can be covered by our hypercontractivity techniques as well, but this requires some technical work which is saved by applying Dona's strong result.

\begin{prop}[\cite{dona2025}, Theorem~1.2]\label{prop:dona}
    There exists $\epsilon>0$ such that the following holds. Let $\sigma^{S_n}, \tau^{S_n}$ be conjugacy classes with size at least $(n!)^{1-\epsilon}$, and $\tau \in A_n$. Then $\tau^{S_n}\subseteq (\sigma^{S_n})^2$.
\end{prop}
\begin{cor}\label{cor:dona}
    There exist $\epsilon_0>0$ such that the following holds. Let $\sigma^{S_n}, \tau^{S_n}$ be conjugacy classes with at most $\epsilon_0\cdot n$ cycles, and $\tau \in A_n$. Than $\tau^{S_n}\subseteq (\sigma^{S_n})^2$.
\end{cor}
Theorem~1.2 in~\cite{dona2025} is stated for conjugacy classes of $A_n$. The stronger statement allowing $\sigma \in S_n$ follows from Reduction I in Section~5 which is proven as an intermediate theorem in the paper. 

We also use the aforementioned well-known result of Bertram~\cite{bertram1972even}.
\begin{lem}[\cite{bertram1972even}, Corollary~2.1] \label{lem:big-cycle-sigma}
    Let $n > 4$ and $\lfloor \frac{3n}{4} \rfloor \le \ell \le n$, and let  $\sigma \in S_n$  be an $\ell$ cycle (with the other $n-\ell$ elements being fixed points). Then ${(\sigma^{S_n})}
    ^2 = A_n$.
\end{lem}

The following results are from~\cite{KLS23}. The first of them is an easy lemma.
\begin{lem}[\cite{KLS23}, Lemma~6.2] \label{lem:cancelation}
Let \(C_1,C_1',C_1''\) be conjugacy classes of \(S_m\), and let
\(C_2,C_2',C_2''\) be conjugacy classes of \(S_{n-m}\). Suppose that
\[
C_1'' \subseteq C_1 C_1'
\qquad\text{and}\qquad
C_2'' \subseteq C_2 C_2'.
\]
Then
\[
C_1'' \oplus C_2''
\subseteq
(C_1\oplus C_2)(C_1'\oplus C_2').
\]
\end{lem}

\begin{lem}[Corollary of Lemma 6.3 of~\cite{KLS23}]\label{lem:2-cases}
For all $b$ and all $a \ge 2$ we have $(b^{2a}) \subseteq (a^{2b})^2$, where $(x^{y})$ is defined as the conjugacy class that consists of $y$ cycles of length $x$. In particular if $(\sigma) = (\sigma') \oplus (a^{2b})$, $\tau = (\tau') \oplus (b^{2a})$, and $(\sigma')^{2} \supseteq \tau',$ then $(\sigma)^{2}\supseteq \tau.$  
\end{lem}

By the lemma, when trying to show $\tau^{A_n}\subseteq (\sigma^{S_n})^2$, we may assume that either $\tau$ has less than $2A$ cycles of length $\ell$ for all $\ell \le B$ or $\sigma$ has less than $2B$ cycles of length $\ell$ for all $1 < \ell \le A$. Indeed, the rest of the cases are covered by the lemma.

\medskip
We use the ``umvirate'' language which is standard in discrete Fourier analysis (see, e.g.,~\cite{KLS23}).

\begin{defn}
    A \emph{dictator} is a set of the form
\[
D_{i\to j}:=\{\pi\in S_n:\ \pi(i)=j\}.
\]
An intersection of (distinct) dictators is called a $t$-\emph{umvirate} if it is nonempty.
Equivalently, a nonempty $t$-umvirate is specified by two $t$-tuples
$I=(i_1,\dots,i_t)$ and $J=(j_1,\dots,j_t)$ with all $(i_1,\dots,i_t)$ and $(j_1,\dots,j_t)$ distinct (but $i_k$ may be equal to $j_\ell$), and is defined as
\[
U_{I\to J}:=\{\pi:\ \pi(i_r)=j_r\ \ \forall r\in[t]\}.
\]
\end{defn}

\begin{defn}[$d$-restriction]
A $d$-restriction of a function is its restriction to a $d$-umvirate $U_{I \to J}$ where $|I|=d$.
    
\end{defn}
\begin{defn}[$k$-chain and $k$-cycle]
 Consider the $k$-umvirate $U_{I \to J}$ where $I=(i_{1}, \dots, i_{k})$ and $J=(i_{2}, \dots, i_{k+1})$ (corresponding to $i_{1} \to i_{2} \to \dots \to i_{k+1}$). We call this restriction a $k$-chain if all coordinates $i_{1}, \dots i_{k+1}$ are distinct. Otherwise, if $i_1 = i_{k+1}$ and all other coordinates are distinct, then we call it a $k$-cycle.
\end{defn}
We use the following lemma.
\begin{lem}[The $2m$-gadget, ~\cite{KLS23}, step 2 of the proof of Lemma 6.1]\label{lem:small-cycle-removal}
Let $m \ge 2$. There exist $2m$-tuples
$S,T,W$ such that:
\begin{enumerate}
\item the restriction pattern $T \to W$ is a disjoint union of $2$ cycles of length $m$;
\item each of the restriction patterns $S \to T$ and $S \to W$ is a disjoint union
of 
\[
\lfloor\frac  m2\rfloor  \text{ many $1$-cycles},\qquad
2m-3\lfloor\frac  m2\rfloor\text{ many $1$-chains},\qquad
\lfloor \frac m2 \rfloor\text{ many $2$-chains}.
\]
\end{enumerate}
\end{lem}

\subsection{Globalness and hypercontractivity.}

For the hypercontractivity component of our proof, we use several results of Ellis, Friedgut and Pilpel~\cite{ellis2011intersecting}, Keevash and Lifshitz~\cite{keevash2023sharp} and~\cite{KLS23}. To present these results, a few more definitions are needed.

\begin{defn}[$r$-globalness]
    A function $f$ is defined to be \emph{$r$-global} if 
    \[
        \|f_{I\to J}\|_2 \le r^{|I|}\|f\|_2
    \]
    for all $d$-restrictions $f_{I\to J}$ and for all every $d\in \mathbb{N}$. A set $A$ is called $r$-global if its indicator function is $r$-global. 
\end{defn}

\begin{defn}[Levels of representations]
    Let $\lambda = (\lambda_1, \lambda_2, \ldots, \lambda_t) \vdash n$ be a partition. The \emph{strict level} of the representation $V_{\lambda}$ of $S_n$ corresponding to $\lambda$ is defined as $n - \lambda_1$. The \emph{level} of $V_{\lambda}$ is the minimum of the strict levels of $V_{\lambda}$ and $V_{\lambda'}$, where $\lambda'$ denotes the conjugate partition of $\lambda$.
\end{defn}

\begin{defn}[Space of Matrix Coefficients]
    The \emph{space of matrix coefficients} of an irreducible representation $V$ of a finite group $G$ is the space spanned by the functions $f_{v,\varphi} \colon G \to \mathbb{C}$, indexed by $v \in V$ and $\varphi \in V^*$, which are given by
    \[
        f_{v,\varphi}(g) = \varphi(g v).
    \]
    We denote this space by $W_\chi$ where $\chi$ is the character, and $f^{= \chi}$ the projection of a function $f$ onto $W_\chi$. We sometimes refer to this space also by \emph{isotypic component}. 
    Similarly, we let $W_d$ denote the sum of the spaces of matrix coefficients for all representations of level $d$, and we denote by $f^{= d}$ the projection of a function $f$ onto $W_d$.
\end{defn}

We use the following result of Keevash and Lifshitz~\cite{keevash2023sharp}.
\begin{thm}[{\cite[Theorem 4.1]{keevash2023sharp}}]\label{thm:level-d for global functions_intro}
There exists $C>0$, such that for any $n \in \mathbb{N}$ and for any $r>1$, if $A \subseteq S_n$ is $r$-global and $d\le \min(\tfrac{1}{8}\log(1/\mu(A)), 10^{-5}n)$, then 
\[\|1_A^{= d}\|_2^2\le \mu(A)^2 \left( C r^4 d^{-1} \log (1/\mu(A)) \right)^d,\]
where $\mu$ is the uniform measure on $S_n$.
\end{thm}
We also use the following lower bounds of Ellis, Friedgut and Pilpel~\cite{ellis2011intersecting} on the dimensions of low level irreducible representations.
\begin{lem}[{\cite[Claim 1 and Theorem 19]{ellis2011intersecting}}]\label{lem:ellis}
There exists $n_0 \in \mathbb{N}$, such that the following holds for all $n>n_0$. Let $d\le n/200$ and let $\chi$ be an irreducible character of $S_n$ of level $\ge d$. Then $\chi(1) \ge \left(\frac{n}{ed}\right)^d.$ 
\end{lem}
In~\cite{KLS23}, the following proposition was deduced from Theorem~\ref{thm:level-d for global functions_intro} and Lemma~\ref{lem:ellis}:
\begin{prop}[\cite{KLS23}, Proposition~3.3]\label{prop:3.3}
    For any $\epsilon>0$ there exist $\delta,n_0>0$, such that the following holds for all $n>n_0$. Let $\alpha<1-\epsilon$ and let $A\subseteq A_n$ be an $n^\delta$-global set of density $\mu_{A_n}(A)\ge e^{-n^{\alpha}}$. Write $g = \frac{1_A}{\mu_{A_n}(A)}$. Then $\|g^{=\chi}\|_2^2\le \chi(1)^{\alpha+\epsilon}$ for any $\chi \in \widehat{A_n}$.
\end{prop}
By the same method, one can prove the following: 
\begin{prop}\label{prop:globalness-bounds}
    For every $C>0$ there exists $n_0>0$, such that the following holds for all $n>n_0$. Let  $A\subseteq A_n$ be an $n^{0.01}$-global set of density $\mu_{A_n}(A)\ge n^{-C}$. Write $g = \frac{1_A}{\mu_{A_n}(A)}$. Then $\|g^{=\chi}\|_2^2\le \chi(1)^{0.05}$ for any $\chi \in \widehat{A_n}$.
\end{prop}
The proof, which is similar to the proof of Proposition~\ref{prop:3.3} in~\cite{KLS23}, is omitted.

\section{Globalness and measure decay of restricted conjugacy classes}
\label{sec:globalness}

Let $A \subseteq S_n$ be a conjugacy class, and let
$U_{I\to J}:=\{\pi\in S_n:\ \pi(i_r)=j_r\ \forall r\}$ 
be a $t$-umvirate, where $I=(i_1,\dots,i_t)$ and $J=(j_1,\dots,j_t)$. This section is devoted to analyzing properties of restricted conjugacy classes of the form $A\cap U_{I\to J}$ that will be used in the proof of the main result. We first present a method for computing the measure of restricted conjugacy classes.

For an umvirate $U_{I \to J}$, define its restriction graph $G$ to be the directed graph with vertex set $I \cup J$ and edges $i_r \to j_r$.

\begin{defn}\label{def:embedding}
For a permutation $\tau\in S_n$, let $D(\tau)$ be the directed graph on $[n]$ with edges
\[
x\to \tau(x)\qquad (x\in[n]).
\]
Given a directed graph $G$ with $|V(G)| \le n$, an \emph{embedding} of $G$ into $D(\tau)$ is an injective map
$f:V(G)\hookrightarrow [n]$ such that for every edge $u\to v$ of $G$ one has
$f(v)=\tau(f(u))$. Let $\Emb_\tau(G)$ denote the set of these embeddings.
\end{defn}

Figure~\ref{fig:embedding_example} illustrates an embedding of a restriction graph $G$ (consisting of a 1-chain and a 2-chain) into a permutation $\tau$. The embedding $f$ maps the vertices of $G$ to the elements of $[n]$ such that the directed edges of $G$ perfectly align with the cyclic action of $\tau$.

\begin{figure}[htbp]
\centering
\begin{tikzpicture}[>=Stealth, scale=1]
    
    \node at (-3, 1.8) {\large Restriction Graph $G$};
    
    \filldraw[blue] (-3.5, 0.8) circle (1.5pt);
    \filldraw[blue] (-2.5, 0.8) circle (1.5pt);
    \draw[->, thick, blue] (-3.5, 0.8) -- (-2.5, 0.8);
    
    \filldraw[blue] (-4, -0.2) circle (1.5pt);
    \filldraw[blue] (-3, -0.2) circle (1.5pt);
    \filldraw[blue] (-2, -0.2) circle (1.5pt);
    \draw[->, thick, blue] (-4, -0.2) -- (-3, -0.2);
    \draw[->, thick, blue] (-3, -0.2) -- (-2, -0.2);

    \node at (2.5, 2.2) {\large Large cycle of $\tau$};
    \draw[thick, gray, dashed] (2.5, 0) circle (1.6cm);
    
    \tikzset{embedded/.style={-{Stealth[length=2.8mm, width=2.3mm]}, ultra thick, blue}}

    \draw[embedded] (2.5,0) ++(60:1.6cm) arc (60:35:1.6cm);
    \filldraw[blue] (2.5,0) ++(60:1.6cm) circle (1.5pt);
    \filldraw[blue] (2.5,0) ++(35:1.6cm) circle (1.5pt);

    \draw[embedded] (2.5,0) ++(210:1.6cm) arc (210:185:1.6cm);
    \draw[embedded] (2.5,0) ++(185:1.6cm) arc (185:160:1.6cm);
    \filldraw[blue] (2.5,0) ++(210:1.6cm) circle (1.5pt);
    \filldraw[blue] (2.5,0) ++(185:1.6cm) circle (1.5pt);
    \filldraw[blue] (2.5,0) ++(160:1.6cm) circle (1.5pt);

    \draw[-{Stealth[length=3.5mm, width=2.5mm]}, dashed, very thick, red] (-1.5, 0.3) to[bend left=12] node[above, font=\large] {$f$} (0.6, 0.3);

\end{tikzpicture}
\caption{An embedding $f \in \Emb_\tau(G)$. The short disjoint chains of $G$ (left) are mapped to small contiguous arcs along the cycles of $\tau$ (right).}
\label{fig:embedding_example}
\end{figure}
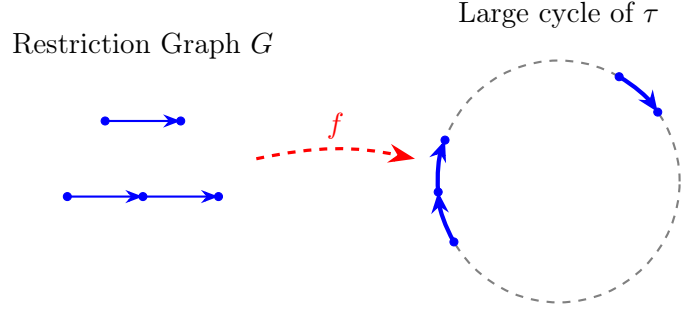

The conditional probability $\mu(A \mid U_{I \to J})$ can be expressed exactly by counting graph embeddings:
\begin{lem}[Exact count under $\mu_A$ via embeddings]\label{lem:embcount}
Let $U_{I \to J}$ be an umvirate, let $G$ be its restriction graph, and let
$m := |V(G)|$ and $k = |E(G)|$. Let $A \subseteq S_n$ be a conjugacy class and fix $\tau \in A$. Then
\[
\frac{\mu(A \mid U_{I \to J})}{\mu(A)} = \frac{|\Emb_\tau(G)|}{(n-k)_{m-k}},
\qquad
(a)_b := a(a-1)\cdots(a-b+1).
\]
\end{lem}

\begin{proof}
By Bayes' rule,
\[
    \frac{\mu(A \mid U_{I \to J})}{\mu(A)} = \frac{\mu(U_{I \to J}\mid A)}{\mu(U_{I \to J})} = \mu(U_{I \to J}\mid A)\cdot (n)_k.
\]
Since $(n)_k \cdot (n-k)_{m-k} = (n)_m$, it remains to show that
\[
\mu(U_{I \to J}\mid A)
=\frac{|\Emb_\tau(G)|}{(n)_m}.
\]

Let $\sigma=\pi^{-1}\tau\pi$ with $\pi$ uniform in $\Sn$, so $\sigma$ is uniform in $A$.
For an edge constraint $\sigma(u)=v$ we have
\[
\pi^{-1}\tau\pi(u)=v \iff \tau(\pi(u))=\pi(v).
\]
Thus, $f:=\pi\!\restriction_{V(G)}$ is a random injective function, and the event $\sigma\in U_q$ is equivalent to $f \in \Emb_\tau(G)$. The statement follows since there are exactly $(n)_m$ injective functions from $[n]$ to $V(G)$.
\end{proof}

Lemma~\ref{lem:embcount} provides an exact expression for the effect of a restriction. We now use this expression to derive lower and upper bounds in the cases needed below.
\begin{lem}[Measure lower bound after restriction]
\label{lem:measure_lower_bound}
For every $C > 0$, set $c := \frac{1}{2 C}$. For every sufficiently large $n$, let $A \subseteq S_n$ be a conjugacy class with $\cyc(A) \le C$, and let $U_{I' \to J'}$ be an umvirate of size $k = |I'| = |J'| < cn$ consisting solely of disjoint chains. Then:
\[
\mu(A \mid U_{I' \to J'}) \ge \frac{1}{2} \mu(A).
\]
\end{lem}

\begin{proof}
By Lemma~\ref{lem:embcount}, for any fixed $\tau \in A$,
\begin{equation}\label{eqn:bayes_lower}
    \frac{\mu(A \mid U_{I' \to J'})}{\mu(A)} = \frac{|\Emb_\tau(G)|}{(n-k)_{m-k}},
\end{equation}
where $G$ is the restriction graph of $U_{I' \to J'}$, $m = |V(G)|$ and $k = |E(G)|$. Since $G$ consists entirely of disjoint chains, it has exactly $N_c := m - k$ chains.

Let $\rho = (1, 2, \dots, n)$ be the canonical $n$-cycle permutation. We call an embedding $f \in \Emb_\rho(G)$ \emph{acyclic} if $f(u) \neq n$ for every directed edge $u \to v$ in $G$. To lower-bound $|\Emb_\rho(G)|$, it suffices to count the number of acyclic embeddings. Under an acyclic embedding, each chain in $G$ with $L_i$ edges ($L_i + 1$ vertices) is mapped to a contiguous integer interval $[x, x + L_i] \subseteq [n]$. Contracting the internal $L_i$ edges of each chain reduces the total required positions from $n$ to $n - \sum L_i = n - k$. Assigning the heads of the $N_c = m - k$ distinct chains to distinct positions among these $n - k$ contracted slots uniquely determines a valid acyclic embedding. Therefore, the number of acyclic embeddings is $(n - k)_{N_c}$, yielding
\[
|\Emb_\rho(G)| \ge (n - k)_{N_c} = (n - k)_{m - k}.
\]

Now fix $\tau \in A$, which has cycle count $R = \cyc(A) \le C$. Without loss of generality, represent $\tau$ as a product of disjoint cycles $\tau = (1 \dots \ell_1)(\ell_1+1 \dots \ell_1+\ell_2)\dots(\dots n)$, where $\ell_1 + \dots + \ell_R = n$. We convert every embedding $f \in \Emb_\rho(G)$ to a potential embedding of $G$ in $D(\tau)$, and bound the probability that it is a valid embedding. See Figure~\ref{fig:lower_bound_shift} for a macroscopic visualization of when the converted embedding is invalid. Formally, for an embedding $f \in \Emb_\rho(G)$ and a uniform random shift $j \sim [n]$, define $f_j := \rho^j \circ f$. Note that $f_j \in \Emb_\rho(G)$ for all $j$ and that $f_j \in \Emb_\tau(G)$ if and only if $\tau(f_j(u)) = f_j(v)$ for all $u \to v \in E(G)$. Since $f_j \in \Emb_\rho(G)$, we have $\rho(f_j(u)) = f_j(v)$ for all $u \to v \in E(G)$. Thus, $f_j \in \Emb_\tau(G)$ holds if and only if $\tau(y) = \rho(y)$ for all $y \in f_j(\Src)$, where $\Src$ is the set of sources of edges in $G$.
\begin{figure}[htbp]
\centering
\begin{tikzpicture}[>=Stealth]
    
    \tikzset{embedded/.style={-{Stealth[length=2.8mm, width=2.3mm]}, ultra thick, blue}}
    \tikzset{broken/.style={ultra thick, orange}}

    \node at (0, 3.5) {\large Sequence $\rho$};
    
    \draw[ultra thick, gray, ->] (-5, 2) -- (5, 2);
    
    \draw[red, thick, dashed] (0, 1.6) -- (0, 2.7) node[above, text=red] {Boundary};
    
    \draw[line width=3mm, blue, opacity=0.5] (-4, 2) -- (-2, 2);
    \node[above, blue] at (-3, 2.2) {Valid};

    \draw[line width=3mm, blue, opacity=0.5] (2, 2) -- (3.5, 2);
    \node[above, blue] at (2.75, 2.2) {Valid};

    \draw[line width=3mm, orange, opacity=0.5] (-0.8, 2) -- (0.8, 2);
    \node[below, orange] at (0, 1.7) {Invalid};

    \draw[->, ultra thick, dashed, gray] (0, 0.8) -- (0, -0.8) node[midway, right, text=black] {Map to $\tau$};

    
    \draw[thick, gray, dashed, ->] (-2.5, -2.8) circle (1.6cm);
    \node[gray] at (-2.5, -2.8) {Cycle 1};
    
    \draw[embedded] (-2.5,-2.8) ++(150:1.6cm) arc (150:90:1.6cm);
    
    \draw[broken, -] (-2.5,-2.8) ++(30:1.6cm) arc (30:0:1.6cm);
    \filldraw[orange] (-2.5,-2.8) ++(0:1.6cm) circle (2pt);

    \draw[thick, gray, dashed, ->] (2.5, -2.8) circle (1.6cm);
    \node[gray] at (2.5, -2.8) {Cycle 2};
    
    \draw[embedded] (2.5,-2.8) ++(90:1.6cm) arc (90:40:1.6cm);
    
    \draw[broken, -] (2.5,-2.8) ++(180:1.6cm) arc (180:140:1.6cm);
    \filldraw[orange] (2.5,-2.8) ++(180:1.6cm) circle (2pt);
    
    \node[orange, font=\small] at (0, -2.8) {Breaks!};

\end{tikzpicture}
\caption{A macroscopic view of the embedding conversion. Embedding chains into the canonical sequence $\rho$ (top) guarantees a valid embedding into $\tau$ (bottom) \emph{unless} a chain crosses one of the $R \le C$ boundaries where $\tau$ separates into disjoint cycles.}
\label{fig:lower_bound_shift}
\end{figure}
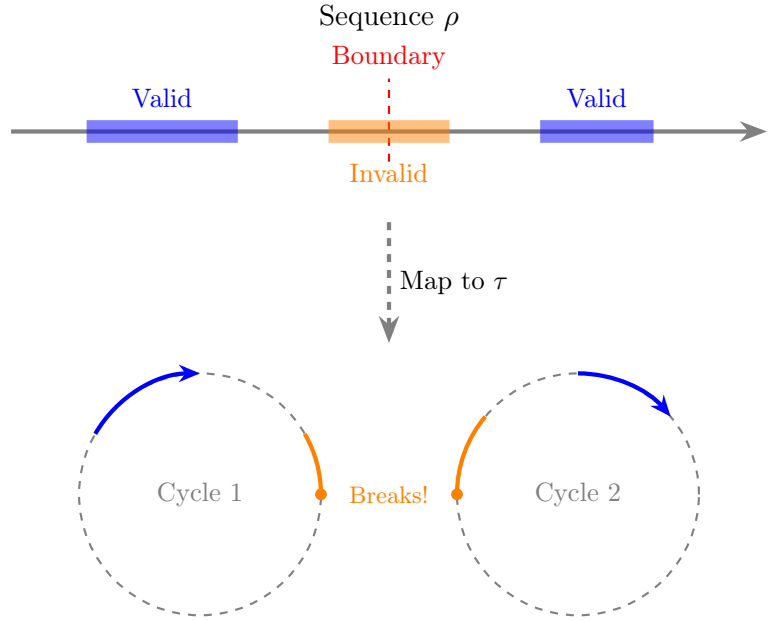

Notice that $\tau(y) = \rho(y) = y+1$ for all $y \in [n]$ except at the boundary element (the last element) of each of the $R \le C$ cycles of $\tau$ (see Figure~\ref{fig:lower_bound_shift}). Thus, there are at most $C$ values of $y \in [n]$ where $\tau(y) \neq \rho(y)$. For any fixed $u \in \Src$, as $j \sim [n]$ is uniform, $f_j(u) = \rho^j(f(u))$ is uniformly distributed over $[n]$. Hence,
\[
\Pr_{j \sim [n]} [\tau(f_j(u)) \neq \rho(f_j(u))] \le \frac{C}{n}.
\]
By the union bound over all $k = |\Src| = |E(G)|$ sources,
\[
\Pr_{j \sim [n]} [f_j \notin \Emb_\tau(G)] \le k \cdot \frac{C}{n} < cn \cdot \frac{C}{n} = c C = \frac{1}{2}.
\]
Consequently, the probability of $\rho^j \circ f \in \Emb_\tau(G)$ for uniform $j \sim [n]$ and $f \sim \Emb_\rho(G)$ is at least $\frac{1}{2}$. Moreover, for every fixed $j$ the map $f \mapsto \rho^j \circ f$ is a bijection on $\Emb_\rho(G)$, implying that the composition $\rho^j \circ f$ induces a uniform distribution on $\Emb_\rho(G)$. We obtain that
\[
\Pr_{f \sim \Emb_\rho(G)} [f \in \Emb_\tau(G)] \ge \frac{1}{2}
\]
and deduce that
\[
|\Emb_\tau(G)| \ge \frac{1}{2} |\Emb_\rho(G)| \ge \frac{1}{2} (n - k)_{m - k}.
\]
Substituting this bound into \eqref{eqn:bayes_lower} yields $\mu(A \mid U_{I' \to J'}) \ge \frac{1}{2} \mu(A)$.
\end{proof}

\begin{lem}[Measure upper bound after restriction]
\label{lem:measure_upper_bound}
For every $C > 0$, there exists $c > 0$ such that the following holds for all sufficiently large $n$. Let $A \subseteq S_n$ be a conjugacy class with $\cyc(A) \le C$, and let $U_{I' \to J'}$ be an umvirate of size $k = |I'| = |J'| < cn$ whose restriction graph $G$ consists of an arbitrary number of disjoint chains of length at most $C$, together with at most $C$ disjoint cycles of length at most $C$ each. Then:
\[
\mu(A \mid U_{I' \to J'}) \le O(\mu(A)).
\]
\end{lem}

\begin{proof}
By Lemma~\ref{lem:embcount}, for any fixed $\tau \in A$,
\begin{equation}\label{eqn:bayes_upper}
    \frac{\mu(A \mid U_{I' \to J'})}{\mu(A)} = \frac{|\Emb_\tau(G)|}{(n-k)_{m-k}},
\end{equation}
where $G$ is the restriction graph of $U_{I' \to J'}$, $m = |V(G)|$ and $k = |E(G)|$.

Decompose $G = G_{\text{chains}} \cup G_{\text{cycles}}$, where $G_{\text{cycles}}$ consists of $R_G \le C$ disjoint cycles $Q_1, \dots, Q_{R_G}$ of lengths at most $C$, and $G_{\text{chains}}$ consists of $N_{\text{chains}}$ disjoint linear chains, each of length at most $C$. Let $k_{\text{cycles}} = |E(G_{\text{cycles}})| = |V(G_{\text{cycles}})| \le C^2$ and $k_{\text{chains}} = |E(G_{\text{chains}})|$. Because each cycle has an equal number of vertices and edges, the number of chains satisfies
\[
N_{\text{chains}} = |V(G_{\text{chains}})| - |E(G_{\text{chains}})| = m - k.
\]

To bound $|\Emb_\tau(G)|$, we first bound the choices for embedding $G_{\text{cycles}}$ into $\tau \in A$: Each cycle $Q_i$ of length $\ell_i \le C$ in $G_{\text{cycles}}$ must map bijectively onto a cycle of $\tau$ of the exact same length $\ell_i$.
Since $\tau$ contains at most $C$ cycles in total, there are at most $C$ candidate cycles in $\tau$ for each $Q_i$, and for each candidate cycle, at most $\ell_i \le C$ possible cyclic alignments.

Hence, the number of valid embeddings of $G_{\text{cycles}}$ into $\tau$ is bounded by at most $(C \cdot C)^{R_G} \le C^{2C} = O(1)$.

Next, fix any valid embedding of $G_{\text{cycles}}$ and consider embedding $G_{\text{chains}}$. Let $Z_1, \dots, Z_R$ ($R \le C$) be the disjoint cycles of $\tau$. Extend each cycle $Z_r$ into a linear sequence $\widetilde{Z}_r$ of length $|Z_r| + C$ by appending $C$ virtual elements at its end, where each virtual element acts as a duplicate of an initial element in $Z_r$. (See Figure~\ref{fig:virtual_elements} for a visualization of how these virtual elements allow boundary-wrapping chains to be mapped contiguously.) Concatenating these sequences yields a single linear sequence $P = \widetilde{Z}_1 \dots \widetilde{Z}_R$ of total length $n' = n + R C \le n + C^2$.
\begin{figure}[htbp]
\centering
\begin{tikzpicture}[>=Stealth]

    \tikzset{curvedchain/.style={-{Stealth[length=2.8mm, width=2.3mm]}, ultra thick, orange}}
    \tikzset{linearchain/.style={line width=3mm, orange, opacity=0.5}}
    \tikzset{normalcurved/.style={-{Stealth[length=2.8mm, width=2.3mm]}, ultra thick, blue}}
    \tikzset{normallinear/.style={line width=3mm, blue, opacity=0.5}}

    \node at (0, 2.7) {\large Cycles of $\tau$};

    \draw[thick, gray, dashed, ->] (-2.5, 1) circle (1.2cm);
    \node[gray] at (-2.5, 1) {$Z_1$};
    \draw[red, thick, dashed] (-2.5, 2.2) -- (-2.5, 2.6); 
    
    \draw[curvedchain] (-2.5,1) ++(120:1.2cm) arc (120:60:1.2cm);
    \node[orange, above] at (-2.5, 2.5) {Wraps boundary};

    \draw[thick, gray, dashed, ->] (2.5, 1) circle (1.2cm);
    \node[gray] at (2.5, 1) {$Z_2$};
    \draw[red, thick, dashed] (2.5, 2.2) -- (2.5, 2.6); 
    
    \draw[normalcurved] (2.5,1) ++(315:1.2cm) arc (315:225:1.2cm);
    \node[blue, below] at (2.5, -0.3) {Inside boundary};

    \draw[->, ultra thick, dashed, gray] (0, -0.4) -- (0, -1.8) node[midway, right, text=black] {Unroll \& Concatenate};

    \node at (0, -2.2) {\large Sequence $P = \widetilde{Z}_1 \widetilde{Z}_2$};

    \draw[ultra thick, gray, ->] (-5.5, -3.5) -- (-1.5, -3.5);
    \draw[ultra thick, gray, dashed, ->] (-1.5, -3.5) -- (0, -3.5);
    \draw[red, thick, dashed] (-1.5, -3.3) -- (-1.5, -3.8); 

    \draw[decorate, decoration={brace, amplitude=4pt}, thick] (-5.5, -3.1) -- (-1.5, -3.1) node[midway, above=4pt] {$Z_1$};
    \draw[decorate, decoration={brace, amplitude=4pt}, thick] (-1.5, -3.1) -- (0, -3.1) node[midway, above=4pt, font=\small] {$C$ virtual};

    \draw[ultra thick, gray, ->] (0, -3.5) -- (4, -3.5);
    \draw[ultra thick, gray, dashed, ->] (4, -3.5) -- (5.5, -3.5);
    \draw[red, thick, dashed] (4, -3.3) -- (4, -3.8); 

    \draw[decorate, decoration={brace, amplitude=4pt}, thick] (0, -3.1) -- (4, -3.1) node[midway, above=4pt] {$Z_2$};
    \draw[decorate, decoration={brace, amplitude=4pt}, thick] (4, -3.1) -- (5.5, -3.1) node[midway, above=4pt, font=\small] {$C$ virtual};

    \draw[linearchain] (-2.2, -3.5) -- (-0.8, -3.5);

    \draw[normallinear] (1, -3.5) -- (2.5, -3.5);

\end{tikzpicture}
\caption{The cycles of $\tau$ are unrolled and concatenated into a single linear sequence $P$. Appending $C$ virtual elements at the end of each cycle ensures that chains crossing cycle boundaries (orange) can be placed contiguously without breaking.}
\label{fig:virtual_elements}
\end{figure}
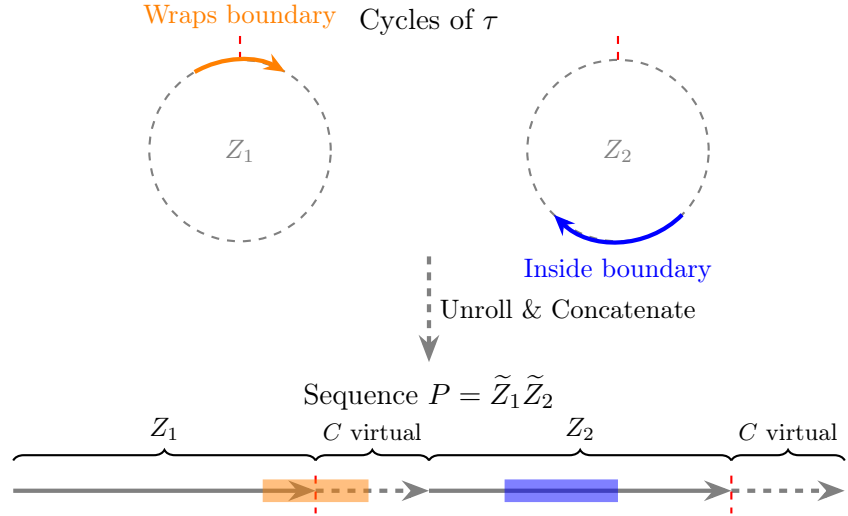

Because every chain in $G_{\text{chains}}$ has length at most $C$, any valid embedding of a chain into $\tau$ corresponds to a unique contiguous placement inside $P$ (using the virtual elements to encode wrap-around). Contracting the internal edges of each chain reduces the total required positions from $n'$ to $n' - k_{\text{chains}}$. Placing the heads of the $N_{\text{chains}} = m - k$ chains into the unassigned positions of $P$ can be done in at most
\[
(n' - k_{\text{chains}})_{N_{\text{chains}}} \le (n + C^2 - k + k_{\text{cycles}})_{m - k} \le (n + 2C^2 - k)_{m - k}
\]
ways. Multiplying by the choices for $G_{\text{cycles}}$ yields:
\[
|\Emb_\tau(G)| \le C^{2C} \cdot (n + 2C^2 - k)_{m - k}.
\]

Substituting this bound into \eqref{eqn:bayes_upper} gives:
\begin{align*}
\frac{\mu(A \mid U_{I' \to J'})}{\mu(A)} 
&\le C^{2C} \cdot \frac{(n + 2C^2 - k)_{m - k}}{(n - k)_{m - k}} = C^{2C} \prod_{j=0}^{m - k - 1} \left( 1 + \frac{2C^2}{n - k - j} \right).
\end{align*}

Set $c = 1/4$. Since $m \le 2k < 2cn = n/2$, we have $n - k - j \ge n - m > n/2$ for all $0 \le j < m - k$. Applying the inequality $1 + x \le e^x$:
\begin{align*}
\prod_{j=0}^{m-k-1} \left(1+\frac{2C^2}{n-k-j}\right)
&\le \exp\left(\sum_{j=0}^{m-k-1}\frac{2C^2}{n/2}\right) \\
&\le \exp\left((m-k)\cdot\frac{4C^2}{n}\right)
\le e^{4cC^2}=e^{C^2}.
\end{align*}

Combining all constants, we conclude:
\[
\frac{\mu(A \mid U_{I' \to J'})}{\mu(A)} \le C^{2C} e^{C^2} = O(1),
\]
which proves $\mu(A \mid U_{I' \to J'}) \le O(\mu(A))$.
\end{proof}

Finally, we use these bounds to prove globalness of the restricted conjugacy classes described below.
\begin{prop}[Globalness of restricted conjugacy classes]
\label{prop:globalness}
For every $C > 0$, there exists $c > 0$ such that the following holds for all sufficiently large $n$. Let $A \subseteq S_n$ be a conjugacy class with $\cyc(A) \le C$, and let $U_{I' \to J'}$ be an umvirate of size $k = |I'| = |J'| < cn$ consisting solely of 1-chains and 2-chains. 
Then $A \cap U_{I' \to J'}$ is $n^{0.01}$-global in $U_{I' \to J'}$.
\end{prop}

\begin{proof}
Set $\gamma = n^{0.01}$. Let $U_{I \to J}$ be an umvirate of size $d = |I| = |J| \ge 1$ compatible with $U_{I' \to J'}$ and let $U_{I'' \to J''} := U_{I \to J} \cap U_{I' \to J'}$ denote the combined restriction event of size $k'' = |I''| = |J''| \le k + d$. We prove the stronger statement:
\[
\frac{\mu(A \mid U_{I \to J} \cap U_{I' \to J'})}{\mu(A \mid U_{I' \to J'})} \le \gamma^d.
\] 

First, because $U_{I' \to J'}$ has size $k < cn$ and consists entirely of disjoint chains of length at most $2$, by picking a sufficiently small $c>0$ we can apply Lemma~\ref{lem:measure_lower_bound} and obtain a lower bound for the denominator:
\[
\mu(A \mid U_{I' \to J'}) \ge \frac{1}{2} \mu(A).
\]
Furthermore, the assumption $\cyc(A) \le C$ implies the classical bound $\mu(A) \ge \frac{1}{C! n^C}$.

Set $D_0 := \lceil 200 (C + 1) \rceil = O(1)$. We partition the proof into two cases depending on the restriction size $d = |I|$:

\paragraph{Case 1: Large restriction ($d \ge D_0$).}
Since $\mu(A \mid U_{I'' \to J''}) \le 1$, we use the trivial upper bound for the ratio:
\[
\frac{\mu(A \mid U_{I'' \to J''})}{\mu(A \mid U_{I' \to J'})} \le \frac{1}{\frac{1}{2}\mu(A)} \le 2 C! n^C.
\]
For $d \ge D_0$, we have $\gamma^d = (n^{0.01})^d \ge n^{0.01 D_0} \ge n^{2(C + 1)} > 2 C! n^C$ for all sufficiently large $n$. Thus, the claim holds trivially in this case.

\paragraph{Case 2: Small restriction ($d < D_0$).}
Let $G''$ be the restriction graph associated with $U_{I'' \to J''}$. Graph $G''$ is obtained by adding $d < D_0$ edges to $G'$, whose components were initially chains of length at most $2$. Consequently, $G''$ consists of:
\begin{itemize}
    \item At most $d < D_0$ disjoint cycles, each of length at most $3 d < 3 D_0$;
    \item Disjoint linear chains, each of length at most $C' := 3 D_0 + 2 = O(1)$.
\end{itemize}

Set $C'' := \max(C, C') = O(1)$, and set $c''$ to be the constant corresponding to $C''$ in Lemma~\ref{lem:measure_upper_bound}. Assuming that $c \le c''/2$ and that $n$ is sufficiently large, the combined restriction size satisfies $k'' = k + d < cn + D_0 < c'' n$. 

Thus, $U_{I'' \to J''}$ satisfies all hypotheses of Lemma~\ref{lem:measure_upper_bound} with constants $C''$ and $c''$. Applying Lemma~\ref{lem:measure_upper_bound} gives:
\[
\mu(A \mid U_{I'' \to J''}) \le O(\mu(A)).
\]
Combining this with our lower bound for $\mu(A \mid U_{I' \to J'})$, we obtain:
\[
\frac{\mu(A \mid U_{I'' \to J''})}{\mu(A \mid U_{I' \to J'})} \le \frac{O(\mu(A))}{\frac{1}{2} \mu(A)} = O(1).
\]
Since $d \ge 1$ and $\gamma = n^{0.01} \to \infty$ as $n \to \infty$, the $O(1)$ upper bound is strictly smaller than $\gamma^1 \le \gamma^d$ for all sufficiently large $n$, completing the proof.
\end{proof}

\section{Proof of Theorem~\ref{thm:main}}
\label{sec:proof}

We first prove a consequence of globalness showing that two sufficiently large global sets must interact through a suitable conjugacy class.

\begin{lem}\label{lem:analytic-covering}
    For small enough constant $\epsilon > 0$ and for every $C>0$, there exists $n_0 > 0$ such that the following holds for all $n > n_0$. Let $\tau \in A_n$ be a permutation without fixed points. Let $A_1, A_2 \subseteq A_n$ be $n^{0.01}$-global sets, both of density at least $n^{-C}$. Then $\tau^{A_n}A_2 \cap A_1 \neq \emptyset$.
\end{lem}

\begin{proof}
    Denote by $f = \frac{1_{A_1}}{\mu_{A_n}(A_1)}$ and $g = \frac{1_{A_2}}{\mu_{A_n}(A_2)}$ the normalized indicator functions of $A_1$ and $A_2$, respectively. Let $h = \frac{1_{\tau^{A_n}}}{\mu_{A_n}(\tau^{A_n})}$ be the normalized indicator function of the conjugacy class of $\tau$ in $A_n$.

    Since both $f$ and $g$ satisfy the hypotheses of Proposition~\ref{prop:globalness-bounds}, applying the proposition to each of them yields
\begin{equation}\label{eqn:global_implies_fourier_bounds}
|f^{=\chi}|_2^2 \le \chi(1)^{0.05},
\qquad
|g^{=\chi}|_2^2 \le \chi(1)^{0.05}
\qquad
\text{for all } \chi \in \widehat{A_n}.
\end{equation}
    
    We wish to show that $\E_{x, y \in A_n}[f(x) h(y) g(y^{-1}x)] > 0$. We can rewrite the expectation as an inner product:
    \[
    \E_{x, y \in A_n} [f(x) h(y) g(y^{-1}x)] = \langle f ,h * g \rangle_{A_n},
    \]
    where the convolution is defined as $(u * v)(z) = \E_{x \in A_n}[u(x) v(x^{-1}z)]$. 

    Since $h$ is a class function in $A_n$, convolution with $h$ acts as a scalar multiplier on each isotypic component of $L^2(A_n)$. Specifically, for any $\chi \in \widehat{A_n}$, the multiplier is $\frac{\chi(\tau)}{\chi(1)}$. Expanding the inner product over irreducible representations yields:
    \begin{equation}\label{eq:fourier_expansion}
        \langle f ,h* g \rangle_{A_n} = \sum_{\chi \in \widehat{A_n}} \frac{\chi(\tau)}{\chi(1)} \langle f^{=\chi}, g^{=\chi} \rangle_{A_n}.
    \end{equation}
    For the trivial character $\chi = \mathrm{triv}$, we have $\chi(\tau) = \chi(1) = 1$. Since $f$ and $g$ are probability densities, $\langle f^{=\mathrm{triv}}, g^{=\mathrm{triv}} \rangle_{A_n} = \E[f]\E[g] = 1$. Thus, the contribution of the trivial character to the sum is exactly $1$.

Because $\tau$ has no fixed points, Theorem~\ref{lem:larsen-shalev-1} implies that for small $\delta > 0$ and sufficiently large $n$, we have $|\chi(\tau)| \le \chi(1)^{1/2 + \delta}$. Therefore, $\left| \frac{\chi(\tau)}{\chi(1)} \right| \le \chi(1)^{-1/2 + \delta}$.
    
    We now  bound the contribution of the non-trivial characters:
    \begin{align*}
    \left| \langle f, h*g \rangle_{A_n} - 1 \right|
    &\le \sum_{\chi \neq \mathrm{triv}}
        \left| \frac{\chi(\tau)}{\chi(1)} \right|
        \cdot \left| \langle f^{=\chi}, g^{=\chi} \rangle_{A_n} \right| 
    \le \sum_{\chi \neq \mathrm{triv}}
        \chi(1)^{-1/2 + \delta}
        \|f^{=\chi}\|_2 \|g^{=\chi}\|_2 \\
    &\le \sum_{\chi \neq \mathrm{triv}}
        \chi(1)^{-1/2 + \delta} \chi(1)^{0.05} 
    \le \sum_{\chi \neq \mathrm{triv}}
        \chi(1)^{-1/3}.
\end{align*}
Where choosing $\delta \le 0.1$.
     The second inequality follows from the Cauchy--Schwarz inequality and the third one follows from substituting \eqref{eqn:global_implies_fourier_bounds} into the sum. 
    
    By the bounds on the Witten zeta function for the alternating group (Theorem~\ref{thm:witten zeta function}), the sum $\sum_{\chi \neq \mathrm{triv}} \chi(1)^{-1/3}$ decays to $0$ as $n \to \infty$. Therefore, for sufficiently large $n$,
    \[
    \left| \langle f ,h* g\rangle_{A_n} - 1 \right| = o(1) < 1,
    \]
    which forces $\langle f ,h* g \rangle_{A_n} > 0$. This implies that there exist $x \in A_1$ and $y \in A_2$ such that $xy^{-1} \in \tau^{A_n}$, completing the proof.
\end{proof}


Now we can prove Theorem~\ref{thm:main}.
\begin{proof}[Proof of Theorem~\ref{thm:main}]
    Let $\epsilon_0$ be the constant from Corollary~\ref{cor:dona}, let $C=1/\epsilon_0$, and let $C_2 > C$ be a constant to be chosen later. Let $c = c(C_2) < \epsilon_0$ be another constant that will be fixed later.
We show that $\tau^{A_n}\subseteq (\sigma^{S_n})^2$ holds for all $\sigma \in S_n$ and $\tau \in A_n$ under the assumption $\cyc(\sigma) <cn$. We will prove the statement for all sufficiently large $n$, since reducing $c$ makes the statement trivial for small values of $n$.

By applying Lemmas~\ref{lem:2-cases} and~\ref{lem:cancelation}, we repeatedly cancel small cycles from $\sigma$ and $\tau$ as follows. Whenever $\tau$ has at least $2C_2$ cycles of length $\ell$ for some $\ell \le 2C$, and $\sigma$ has at least $4C$ cycles of length $\ell'$ for some $2 \le \ell' \le C_2$, we can apply Lemma~\ref{lem:2-cases} with $a=\ell'$ and $b=\ell$ to cancel some of these cycles. Therefore, we can reduce the permutations $\sigma$ and $\tau$ into a pair $\sigma_1 \in S_{n_1}, \tau_1 \in A_{n_1}$ satisfying one of the following two cases: 
\begin{description}
    \item[Case 1] $\tau_1$ has fewer than $2C_2$ cycles of length $\ell$ for each $\ell \le 2C$; or
    \item[Case 2] $\sigma_1$ has fewer than $4C$ cycles of length $\ell$ for each $2 \le \ell \le C_2$.
\end{description}

Note that the average cycle length of $\sigma$ is $n/\cyc(\sigma)>c^{-1}$, and we canceled cycles of length at most $C_2$. Therefore, assuming $c(C_2) <C_2^{-1}$, the average cycle length of $\sigma_1$ is larger than that of $\sigma$, that is, $n_1/\cyc(\sigma_1)>c^{-1}$. Additionally, assuming $c(C_2) <C_2^{-1}/2$, the number of canceled elements is at most $\cyc(\sigma)\cdot C_2 < n/2$, and hence $n_1 > n/2$ can be assumed to be sufficiently large.

Case 1 follows from Corollary~\ref{cor:dona} assuming $n_1$ is sufficiently large, since $\cyc(\sigma_1) < cn_1 < \epsilon_0 n_1$ and $\cyc(\tau_1) < O(1) + \frac{n_1}{2C} <\epsilon_0 n_1$.
Therefore, we may focus on Case 2 and assume that for every $2 \le \ell \le C_2$, $\sigma_1$ has fewer than $4C$ cycles of length $\ell$. Note that $\sigma_1$ may still contain up to $cn_1$ fixed points. If both $\sigma_1$ and $\tau_1$ contain fixed points, we cancel fixed points from both of them until we obtain a new pair $\sigma_2 \in S_{n_2}, \tau_2 \in A_{n_2}$, such that one of them has no fixed points left. Since we removed only cycles of length 1, which is smaller than the average cycle length of $\sigma_1$, we may deduce that $\cyc(\sigma_2) <cn_2$. We also have $n_1 -n_2 \le \cyc(\sigma_1) < cn_1$, implying that $n_2 > (1-c)n_1$ can be assumed to be sufficiently large. Next, if $\sigma_2$ has no fixed points, then by Lemma~\ref{thm:Larsen-shalev-small-cycles}, $E(\sigma_2)\le 1/C_2 + O(1/\log(n_2))$, and we finish by Theorem~\ref{thm:Larsen-Shlaev} assuming that $C_2 > 4$ and that $n_2$ is sufficiently large. Otherwise, we may assume that $\tau_2$ has no fixed points.

Next, we want to apply Lemma~\ref{lem:big-cycle-sigma} iteratively in order to cancel cycles of $\sigma_2$ and $\tau_2$. To avoid ambiguity as the permutations shrink, let $m$ denote the number of active coordinates at any given step (initially $m = n_2$), and let $\tilde{\sigma} \in S_m$ and $\tilde{\tau} \in A_m$ denote the remaining sub-permutations (initially $\sigma_2$ and $\tau_2$). Let $k$ be the total number of elements in $\tilde{\tau}$ belonging to ``small cycles'' (length $\le C_2/100$). Let $M$ be the set of cycles in $\tilde{\sigma}$ of length between $2$ and $C_2$. We temporarily ignore this set and do not cancel its cycles.

We proceed with the following iterative cancellation step as long as $k > \epsilon_0 m$ and $\tilde{\sigma}$ contains a cycle $C_{big} \notin M$ of length $C_2 < \ell \le \frac{k}{2} - \frac{C_2}{50}$ alongside at least $\ell/3$ fixed points:
We select a collection of small cycles in $\tilde{\tau}$ whose total length $L$ satisfies $\lfloor \frac{4\ell}{3} \rfloor - \frac{C_2}{50} \le L \le \lfloor \frac{4\ell}{3} \rfloor$. We can easily hit a sum exactly in this narrow range because the available small cycles of $\tilde{\tau}$ are all of length at most $C_2/100$, and the sum of their lengths is  $k > 4\ell/3$. (We may also assume that the selected cycles of $\tilde{\tau}$ form an even permutation, by swapping or adding one small cycle if a parity correction is needed, which is safely absorbed by the bounds). 

By Lemma~\ref{lem:big-cycle-sigma}, the square of the conjugacy class of an $\ell$-cycle in $S_L$ covers $A_L$. Thus, we can cancel $C_{big}$ and $L - \ell$ fixed points from $\tilde{\sigma}$ and cancel the collection of small cycles from $\tilde{\tau}$. Then we update $m$, update $k$, and repeat the process.

Because $m$ strictly decreases, this iteration must eventually terminate. Crucially, at each step we remove $L$ coordinates from the active set while consuming $L - \ell$ fixed points from $\tilde{\sigma}$. Since $L \ge \lfloor \frac{4\ell}{3} \rfloor - \frac{C_2}{50}$ and $\ell > C_2$, we have $L - \ell \ge \frac{\ell}{3} - \frac{C_2}{50} - 1 \ge \frac{L}{5}$ (assuming $C_2$ is sufficiently large). The total number of fixed points initially available in $\sigma_2$ was at most $c n_2$, meaning the total number of coordinates removed across all iterations cannot exceed $5 c n_2$. Therefore, we have $m \ge (1-5c)n_2$, which guarantees that $m$ remains sufficiently large for the remainder of the proof. 

Furthermore, we can maintain the strong bound $\cyc(\tilde{\sigma}) < cm$ throughout the entire process. This is equivalent to showing that the average cycle length $m/\cyc(\tilde{\sigma})$ remains strictly greater than $c^{-1}$. The initial average length was $n_2/\cyc(\sigma_2) > c^{-1}$, so it suffices to show that this averages increases with every cancellation. At each iterative step, we remove $L$ elements distributed across exactly $1 + L - \ell$ cycles of $\tilde{\sigma}$ (specifically, one cycle of length $\ell$ and $L - \ell$ fixed points). The average length of these removed cycles is $\frac{L}{1 + L - \ell}$. Since $L - \ell \ge \frac{L}{5}$, this removed average is strictly less than $5$. Therefore, assuming $c < 1/5$, the average length of the removed cycles is smaller than the average before the cancellation, implying that every iteration of the cancellation can only increase the average cycle length.

The iteration halts when we can no longer perform the cancellation step, leaving us in one of the following three scenarios:

\begin{enumerate}[label=(\roman*)]
    \item \textbf{$k$ becomes small ($k \le \epsilon_0 m$):}
    Since $\tilde{\tau} \in A_m$ is fixed-point free and has $k$ elements in cycles of length $\le C_2/100$, we have
    \[
        \cyc(\tilde{\tau}) \le \frac{k}{2} + \frac{m}{C_2/100} \le \frac{\epsilon_0}{2} m + \frac{100}{C_2} m \le \epsilon_0 m
    \]
        (choosing $C_2$ large enough). Since we also have $\cyc(\tilde{\sigma}) < c m \le \epsilon_0 m$, we may apply Corollary~\ref{cor:dona} and deduce that $\tilde{\tau} \in (\tilde{\sigma}^{A_m})^2$.

    \item \textbf{$\tilde{\sigma}$ has few fixed points:}
    Suppose the iteration halts because $\tilde{\sigma}$ contains a cycle $C_{big} \in \tilde{\sigma} \setminus M$ of length $C_2 < \ell \le \frac{k}{2} - \frac{C_2}{50}$, but has insufficient fixed points to satisfy the padding ratio requirement (i.e., $F := |\fix(\tilde{\sigma})| < \ell/3$). We analyze two subcases based on $F$:

    If $F < C_2/50$, then $\tilde{\sigma}$ contains at most $O(1)$ elements in cycles of length $\le C_2$ (namely, the $F$ fixed points and the fixed set $M$ of small cycles). Consequently, by Lemma~\ref{thm:Larsen-shalev-small-cycles} $E(\tilde{\sigma}) < \frac{1}{C_2} + \epsilon$, and we finish by Theorem~\ref{thm:Larsen-Shlaev}.

    Otherwise, $F \ge C_2/50$. In this case, we perform one final cancellation using $C_{big}$ together with almost all remaining fixed points. Specifically, instead of requiring $L - \ell \approx \ell/3$, we choose a set of small cycles in $\tilde{\tau}$ forming an even permutation such that the padding $L - \ell$ lies in the interval $[F - \frac{C_2}{50}, F]$. Because this interval has length $C_2/50$, such a selection is guaranteed to exist using the available small cycles of $\tilde{\tau}$. Canceling these elements leaves at most $C_2/50$ fixed points in the updated sub-permutation $\tilde{\sigma}'$.

    To safely apply Theorem~\ref{thm:Larsen-Shlaev} after this final step, we must verify that the number $m'$ of remaining active coordinates is sufficiently large. Indeed, since $\ell \le \frac{k}{2} - \frac{C_2}{50} < \frac{m}{2}$ and the consumed fixed points satisfy $L - \ell \le F < \ell/3 < \frac{m}{6}$, the total number of removed coordinates is at most $\frac{m}{2} + \frac{m}{6} = \frac{2}{3}m$. Thus, $m' \ge \frac{1}{3}m \ge \frac{1}{6}n_2$, which remains sufficiently large. The resulting sub-permutation $\tilde{\sigma}'$ now has $O(1)$ elements in cycles of length $\le C_2$ out of the $m' \ge \frac{1}{6}n_2$ elements, yielding $E(\tilde{\sigma}') < \frac{1}{C_2} + \epsilon$ by Lemma~\ref{thm:Larsen-shalev-small-cycles}, allowing us to finish by Theorem~\ref{thm:Larsen-Shlaev} as above.
\item \textbf{All nontrivial cycles outside $M$ are long:}
    We are left with $k>\epsilon_0m$, and every nontrivial cycle of
    $\tilde{\sigma}$ outside $M$ has length greater than
    $k/2-C_2/50\ge\epsilon_0m/4$. There are therefore at most
    $4/\epsilon_0$ such cycles. There is also at most a constant number of cycles in $M$, since there are fewer than $4C$ cycles of each length from $2$ to $C_2$.
    We now try to `cancel' fixed points of $\tilde{\sigma}$.

    We follow the restriction-and-shift argument of
    \cite[proof of Lemma~6.1, Steps~2--4]{KLS23}. Write
    \[
       F=|\fix(\tilde{\sigma})|,\qquad
       \mathcal I=\tilde{\sigma}^{S_m},\qquad
       \mathcal A=\tilde{\tau}^{S_m}.
    \]
    For a set $\mathcal B$ of permutations, we write
    $\mathcal B_{S\to T}:=\mathcal B\cap U_{S\to T}$.

    \medskip
    \noindent\emph{Removing fixed points by restrictions.}
    Put $L=\lfloor C_2/100\rfloor$. Since the short cycles of
    $\tilde{\tau}$ contain $k>\epsilon_0m$ coordinates, some length
    $r\in\{2,\ldots,L\}$ occurs at least $k/L^2$ times. We choose $c$
    to be small enough such that this number exceeds $2F$. 
    Set $s= \lfloor F/\lfloor r/2\rfloor \rfloor$,   so we can therefore select $2s$ distinct $r$-cycles of
    $\tilde{\tau}$ and apply Lemma~\ref{lem:small-cycle-removal} repeatedly $s$ times on disjoint coordinates. Concatenating their tuples
    gives $S,T,W$ of length $t=2rs$.

    The restriction $T\to W$ corresponds to the selected $2s$ cycles of
    $\tilde{\tau}$. Each of $S\to T$ and $S\to W$ corresponds to
    $sr/2$ fixed points of $\tilde{\sigma}$, together with disjoint
    $1$-chains and $2$-chains. In particular only a bounded number of fixed points remain in $\tilde{\sigma}$, and the chain
    restrictions have at most $5cm$ edges. The coordinates of restrictions that correspond to the different applications of Lemma~\ref{lem:small-cycle-removal} can be chosen to be disjoint when
    $c<1/20$.

    \medskip
    \noindent\emph{Passing to the remaining coordinates.}
    The tuples $T$ and $W$ have the same underlying set. Let $q: [m] \to [m]$ be a bijection sending
    $W$ to $T$ coordinate-wise and fix its complement. It is a product
    of $2s$ disjoint $r$-cycles, so it is even. Choose a permutation
    $p$ sending $W$ to $S$ coordinatewise, with
    $\operatorname{sgn}(p)=\operatorname{sgn}(\tilde{\sigma})$.
     Define
    \[
    \begin{aligned}
       \mathcal I_1=\mathcal I_{S\to W}p, \qquad
       \mathcal I_2=q^{-1}\mathcal I_{S\to T}p,\qquad
       \mathcal A'&=\mathcal A_{T\to W}q.
    \end{aligned}
    \]
    Every permutation in these sets fixes $W$ pointwise. We identify
    the complement of $W$ with $[N]$, where
    \(
       N\ge m/2.
    \)
    This give $\mathcal I_1,\mathcal I_2\subseteq A_N$.
    Moreover, $\mathcal A'=(\tau')^{S_N}$, where $\tau'$ is obtained
    from $\tilde{\tau}$ by deleting the selected pairs of
    $r$-cycles.

    It suffices to prove that
    $\mathcal A'\mathcal I_2\cap\mathcal I_1\ne\varnothing$.
    Indeed, such an intersection gives
    $z\in\mathcal A_{T\to W}$, $x\in\mathcal I_{S\to T}$ and
    $y\in\mathcal I_{S\to W}$ with
    \[
       (zq)(q^{-1}xp)=yp.
    \]
    Hence $z=yx^{-1}\in\mathcal I\mathcal I^{-1}=\mathcal I^2$, which implies $\tilde{\tau}\in\mathcal I^2$.

    \medskip
    We move towards applying Lemma~\ref{lem:analytic-covering}, in order to deduce $\mathcal A'\mathcal I_2\cap\mathcal I_1\ne\varnothing$ and complete the proof. First, note that in each restriction of $\mathcal I$, we can take the fixed points and delete their coordinates. The resulting conjugacy
    class $\overline{\mathcal I}\subseteq S_u$ for $u \ge N$ has
    at most
    \(
       B:=\left\lceil4CC_2+4/\epsilon_0+L\right\rceil
    \)
    cycles. Indeed, this quantity includes the cycles in $M$, the long cycles, and the
    fewer than $L$ remaining fixed points. The remaining restriction
    consists solely of disjoint $1$-chains and $2$-chains, with
    $\le 5cm$ edges. Since $B$ depends only on $C_2$ and
    $\epsilon_0$, we may choose $c$ small enough for both
    Lemma~\ref{lem:measure_lower_bound} and
    Proposition~\ref{prop:globalness} to apply to this restriction
    of $\overline{\mathcal I}$.

    The measure lower bound and the centralizer formula for a
    conjugacy class give, for $i=1,2$,
    \[
       \mu_{S_N}(\mathcal I_i)
       \ge\frac12\mu_{S_u}(\overline{\mathcal I})
       \ge\frac{1}{2B!\,u^B}
       \ge N^{-(B+1)}
    \]
    for sufficiently large $m$, since $u\le m\le2N$. Here the shifts
    identify each restriction's ambient umvirate with $S_N$, so its
    conditional density becomes the ordinary density in $S_N$.
    Passing to $A_N$ doubles these densities.

    We get that $\mathcal{I}_1,\mathcal{I}_2$ are $N^{0.01}$-global from
    Proposition~\ref{prop:globalness}, since shifting preserve globalness, and their density is sufficiently large as mentioned above. Therefore, since $\tau'$ is fixed-point-free, we can apply Lemma~\ref{lem:analytic-covering} with $C=B+1$,
    to obtain that
    \(
       (\tau')^{A_N}\mathcal I_2\cap\mathcal I_1\ne\varnothing
    \)
     for large enough $N$, as needed.

\end{enumerate}
\end{proof}

\section*{Acknowledgments}
During the preparation of this manuscript, the authors used OpenAI's ChatGPT 5.4 to formulate Definition~\ref{def:embedding} and Lemma~\ref{lem:embcount} and to assist in the proof of other lemmas in Section~\ref{sec:globalness}. In addition, the authors used ChatGPT 5.4 and Google's Gemini 3.1 for proofreading, improving the readability and exposition of proofs, and clarifying minor intermediate details, and ChatGPT Astra to rewrite the last part of the proof of the main theorem in a clearer way, and for proof reading. All core mathematical results and proof strategies in this manuscript were obtained entirely by the authors.

\bibliographystyle{alpha}
\bibliography{refs}

\end{document}